\documentclass[a4paper,11pt]{article}
\pdfoutput=1

\usepackage[utf8]{inputenc}
\usepackage{utf8}

\title{Multisets in Type Theory}
\author{Håkon Robbestad Gylterud}

\usepackage[hyphens]{url}
\usepackage{fancyvrb}
\usepackage{hyperref}

\usepackage{lmodern}

\usepackage[backend=bibtex,
            style=authoryear-icomp,
            hyperref=true,
            backref=false,
            maxcitenames=3
            ]{biblatex} % load the package

\usepackage{amssymb,amsmath,amsthm}

\usepackage{stmaryrd}

\usepackage[all]{xy}

\usepackage{chngcntr}
\counterwithout*{paragraph}{subsubsection}
\counterwithin*{paragraph}{section}

\newtheorem{proposition}[paragraph]{Proposition}
\newtheorem{theorem}[paragraph]{Theorem}
\newtheorem{lemma}[paragraph]{Lemma}

\theoremstyle{definition}
\newtheorem{definition}[paragraph]{Definition}

\newcounter{example}

\newenvironment{example}{
\paragraph{Example}
\addtocounter{paragraph}{1}}{\vspace{3mm}}

\usepackage[explicit]{titlesec}

\renewcommand{\thesection}{\arabic{section}}

\renewcommand{\theparagraph}  {\thesection:\arabic{paragraph}}

\titleformat{\paragraph}[runin]{\normalfont\bfseries}{}{0em}{#1\ \theparagraph.}

\usepackage{multisets/notation}

\begin{document}

\maketitle

  \begin{abstract}
A multiset consists of elements, but the notion of a multiset is distinguished
from that of a set by carrying information of how many times each element
occurs in a given multiset. In this work we will investigate the notion of
iterative multisets, where multisets are iteratively built up from other
multisets, in the context Martin-Löf Type Theory, in the presence of
Voevodsky’s Univalence Axiom.

\cite{aczel1978} introduced a model of constructive set theory in type theory,
using a W-type quantifying over a universe, and an inductively defined
equivalence relation on it. Our investigation takes this W-type and instead
considers the identity type on it, which can be computed from the Univalence
Axiom. Our thesis is that this gives a model multisets. In order to demonstrate
this, we adapt axioms of constructive set theory to multisets, and show that
they hold for our model.
\end{abstract}

  \newpage
  \tableofcontents
  \newpage

  \section{Introduction}

The purpose of this paper is to describe a model of iterative,
transfinite multisets and to discuss a possible axiomatisation of the
model in the context of univalent Martin-Löf style type theory. Before
describing the model, we discuss existing work on multisets and their
relation to the model at hand.

Usage of multisets has a long history, both in mathematics and in
applications. In classical mathematics one models multisets inside set
theory in various ways. Here follows a brief description of three common
ways of representing multisets.

A very general definition, introduced in \cite{Rado1975}, is that a
multiset on a domain set $X$, consists of an assignment
$X → \text{Card}$, which for each element of the domain specifies the
(possibly transfinite) number of occurrences of the element in the
multiset. Often, this notion is restricted to functions $X → \N$,
which represent multisets where each element occurs finitely many times.

One can also view a multiset as a set $A$ with an equivalence relation
$R$ defined on $A$. The idea is that the elements of $A$ are the
occurrences in the multiset, and the relation $R$ specifies which
occurrences are the same. Thus the number of occurrences of $a \in A$ is
the size of the $R$-equivalence class of $a$.

A third way is to consider a multiset as a family of sets. The index set
of the family corresponds to the domain in Rado's multisets, but instead
of assigning a cardinal number, we have a set of occurrences.

These three approaches illuminate different aspects of multisets, and
even though they are formulated quite differently it is relatively easy
to pass back and forth between them. In fact they would be equivalent if
one removes the constraint that the relation in the second formulation
should be reflexive, or restricts the other two to ensure that each
element in the domain occurs at least once.

Rado's formulation reflect that elements in a multiset occurs a specific
number of times. This can be problematic in a constructive context,
where the notion of cardinality is much more nuanced. This is solved if
one takes the family-of-sets definition, which makes perfect sense
constructively, but requires more thought as to what constitutes
equality between multisets. The notion of equality between multisets is
a topic we will come back to later in this paper.

Considering a multiset as a set with equivalence relation, a setoid,
gives an interesting way to talk about the different between identical
elements and equal elements. The identity of elements in the underlying
set $A$ tells us when occurrences are identical, and the relation $R$
tells us which occurrences are equal. Since the underlying theory is set
theory, we can distinguish equal occurrences in a multiset, but not
identical occurrences.

All three notions describe what we in this paper will refer to as
``flat multisets'', as opposed to ``iterative multisets''. In a flat
multiset, the elements are taken from some domain which may not consist
of other multisets. The iterative multisets have elements which are
multisets themselves, and the collection iterative multisets is
generated in a well-founded manner.

\cite{blizard1988} develops an axiomatisation of iterative
multisets with finite occurrences. The theory is a two-sorted, first-order
theory. The two sorts $N$ and $M$, represent the natural numbers and
multisets respectively. The natural numbers are given by the Peano axioms.
Membership is interpreted as a ternary predicate $-\in_{-}-$, where the
intended interpretation of $x \in_n y$ is that $x$ occurs exactly $n$
times in $y$. The axioms for multisets are then chosen so that one can
reconstruct $ZFC$ internally as the multisets where each element occurs
at most once.

In this paper, we will take a different view on elementhood of multisets
compared to \cite{blizard1988}. Instead of a ternary relation, we will
keep the ${\in}$-relation binary and invoke the propositions-as-sets
attitude of Martin-Löf type theory.

In Martin-Löf type theory one does not generally distinguish the notion
of set from the notion of proposition. The notion of a set, as given by
its canonical elements, corresponds exactly to the notion of a
proposition as given by its canonical proofs. This leads us to give the
binary relation ${\in}$ the typing ${\in} : M → M → \set$. Thus, for
given $x,y : M$, we have that $x \in y$ is a set.
\emph{The natural interpretation is that $x \in y$ is the set of 
occurrences of $x$ in $y$}.

Taking the idea of using types to capture the number of occurrences
further, we need a notion of equivalence of type. This is where
Univalent Type Theory enters the picture. Voevodsky’s univalence
axiom expresses\footcite[gives an exposition.]{awodey2013} that an
identity between type is exactly (equivalent to) an equivalence.
For types which are mere sets this means that two mere sets are
identified if there is a bijection between them.

The model of multisets which we present in this paper is derived from a
model of constructive set theory by \cite{aczel1978}. The
multiset model can in fact be seen as a description of what Aczel's
model looks like through the eyes of type theory with the Univalence
Axiom.

  \section{Notation and background}

The following article is set with type theory as its intended metatheory.
Some results depend on the Univalence Axiom, and are marked as such.
Part of the article is formalised in Agda \footcite[in the references
contains a URL to the source code of the formalisation.]{agdalib},
in particular the more technical lemmas leading up to the extensionality
theorem and the extensionality theorem itself. However, this article is
self contained, and even the proofs for which there exists a formalisation
are here presented in usual mathematical writing.

In ways of notation we will mostly use standard type theoretical
notation, but written out as informal constructive mathematics in the
style of the homotopy type theory book\footcite{hottbook},
rather than giving formal derivations. Our notation deviate a bit from
the book in ways mentioned below.
Mostly, these deviations take us close to how type theory is written in
a formal proof system, such as Agda or Coq.

Application of functions and instantiation of dependent types are denoted
by juxtaposition\,—\,e.g. $f␣a$ or $B␣a$\,—\,leaving a small space between the
function or type and their argument.

We use $A : \set$ to denote that $A$ is a type. 

Definitions are signified by $≔$ with the type of the term often listed
on the line above the definition. Definitional equalities are denoted by $≡$.

The equality sign $=$ is used to denote various equivalence relations — each time
identified with a subscript, unless clear from the context. We will use the notation
$\Id$ for the identity type.

We follow the book\footcite{hottbook} in the definition of equivalence of types, $A ≃ B$,
homotopy of functions, $f \sim g$,
and notions such as \emph{contractible}, \emph{mere proposition}, 
\emph{mere set}, and $n$-type.
For the basic properties of these notions we refer the reader to (\cite{hottbook}).

Many proofs involve showing
equivalences of types, which we strive to demonstrate, as far as possible,
using chains of simpler type-algebra equivalences, such as
$\left(∏_{a:A}∑_{b:B␣a} C␣a␣b\right) ≃ \left(∑_{f : ∏_{a:A}B␣a}∏_{a:A}C␣a␣(f␣a)\right)$.

It is is worth noting that we will consider quantifiers, such as ∀,∃,∏ and ∑ to bind
weakly, so that for instance $∏_{x:a} P␣x → Q␣x$ disambiguates to  $∏_{x:a} (P␣x → Q␣x)$
rather than $\left(∏_{x:a} P␣x \right) → Q␣x$. We sometimes will add the parenthesis to
emphasise this.

% \begin{enumerate}
%   \item Transitivity of ${\cong}$.
%   \item Equivalence given by substitution.
%   \item Subsitution using ${\cong}$.
%   \item ``Theorem'' of choice.
%   \item Double $\sum$
% \end{enumerate}

  \section{The model}

In the section we recall Aczel's model of constructive set theory, delve into homotopy
type theory and construct a model of multiset theory.

\subsection{Aczel's model}

The idea behind the construction of Aczel's $V$ type in \cite{aczel1978} is
that, given a universe $U : \set$ with decoding type $T : U → \set$, one can
construct a setoid which captures the iteratively generated sets, where each set
has an index of its elements in $U$. An element can be listed more than once in
the index of the set. The equality relation of the setoid removes the
distinction between equal sets with different representations, making sets equal
if they have the same elements.

\begin{definition}
Given a universe $U$, with decoding function $T : U → \set$, let Aczel's $V$ be
the setoid defined as follows.

\begin{align*}
&V : \set \\
&V := W_{a : U} Ta \\
\\
&=_V : V → V → \set \\
&(\sup\ a\ f) =_V (\sup\ b\ g) := \left(∏_{i : T a} ∑_{j : T b}  (f␣i) =_V (g␣j)\right) ∧ \left( ∏_{j : T b} ∑_{i : T a} (f␣i) =_V (g␣j) \right)
\end{align*}
\end{definition}

The way to look at a canonical element $x ≡ (\sup\ a\ f) : V$ is that $a$ is a
code for the index of elements in $x$ and $f : T a → V$ picks out the elements
(i.e. sets) contained in $x$.

\paragraph{Remark.} Notice that the relation $=_V$ is $U$-small, if $U$ has
$Σ$-types and $Π$-types. That is, one can prove by W-induction that for each
pair of elements $x,y : V$, there is a code in $U$ for the type $x =_V y$. This
is important, since we want to use equality to construct indices for new sets.

\begin{definition}
Let elementhood in Aczel's V be defined as follows.

\begin{align}
&{∈} : V → V → \set \\
&x ∈ (\sup␣b␣g) = ∑_{i : Tb} x =_V (g␣i)
\end{align}
\end{definition}

\paragraph{Remark.}

Since the relation \(=_V\) is \(U\)-small, it follows that the relation
\({∈}\) is also \(U\)-small.

\cite{aczel1978} goes on to prove that the setoid \((V,=_V)\),
with the relation \({∈}\), is a model of Constructive Zermelo-Fraenkel
set theory (CZF). In this paper, we take a different path, and ask the
question: What is the nature of elements in \(V\), without taking the
quotient of \(=_V\), and instead considering the identity type on \(V\)?

As noted, a set in \(V\) may have the same element listed several times,
but the equality \(=_V\) erases the distinction between representations
that just differ by the number of times they repeat elements. However we
cannot expect the identity type to do the same. Thus, we expect that the
result will be more like multisets, possibly with the obstacle that
order of elements may play a role. As we will see, this obstacle is
overcome by the \emph{univalence axiom}.

\subsection{The identity type on W-types}

The following result is due to Nils Anders Danielsson
\footcite[only available on-line.]{danielsson2012}.
The result characterises the W-type of a type family \(B : A → \set\),
in terms of the identity type of \(A\) and \(B\), up to equivalence. The
lemma does not make use of the Univalence axiom and can be carried out
in plain Martin-Löf type theory.

A technical detail is that the proof as it stands, relies on
\(η\)-reduction. The justification for this is that we take the function
type in the W-type to be the \(Π\)-type of the logical framework, in
which \(η\)-reduction is customary \footcite{mltt}. This is
also how it is implemented in Agda. However one can carry out the proof
without appeal to the \(η\)-reduction, as \(η\)-reduction holds up to
provable equality (See page 62 of \cite{nps-prog}).

\begin{definition}
Given $A : \set$ and $B : A → \set$, and an element $x : W_AB$ we denote by
$\bar x : A$, and $\tilde x : B \bar x → W_AB$ the operations given by
$\overline {(\sup␣ a␣ f)} ≡ a$ and $\widetilde {(\sup␣ a␣ f)} ≡ f$.
\end{definition}

\begin{lemma}
\label{w-id}
For any $A : \set$ and $B : A → \set$, and all $x, y : W_A B$, there is an
equivalence

\begin{align*}
 \Id_{W_A B} ␣x␣ y ≃ ∑_{α : \Id_A ␣ \overline{x}␣\overline{y}} ␣ \Id ␣\widetilde{x}  ␣ (Bα · \widetilde{y})
\end{align*}
\end{lemma}

\subsection{A model of multisets}

We will now present our model of transfinite, iterative multisets, given a
univalent universe $U : \set$ with decoding function $T : U → \set$. It consists
of a type $M$ of multisets, an equality relation $=_M$ and a relation ${∈}$,
which expresses elementhood. The type $M$ is the same $W$-type as Aczel's $V$.
The equality, however, is logically stricter than Aczel's equality, and, as we
will show, equivalent in a strong sense to the identity type of $M$.

\begin{definition}
 We define
\begin{align*}
&M:\set \text{ by }\\
&M:= W_{a : U} Ta 
\end{align*}
and
\begin{align*}
&=_M : M → M → \set \text{ by }\\
&(\sup␣ a␣ f)=_M (\sup␣ b␣ g) := ∑_{α : T a ≃ T b} ∏_{x: Ta} (fx)=_M(g ␣(α␣ x))
\end{align*}
and
\begin{align*}
&{∈} : M → M → \set \text{ by }\\
&x ∈ (\sup b␣ g) := ∑_{i : Tb} x=_M (g␣ i),
\end{align*}
% where ${≃}$ denotes equivalence of types as defined in the homotopy hype theory book\footcite{hottbook}.
\end{definition}

\paragraph{Remark:} Observe that if $U$ has $Π$-types, $Σ$-types and identity
types, then $=_M$ and $∈$ are $U$-small, just like their corresponding relations
in Aczel's $V$.

\subsection{Equality and the identity type}

\begin{theorem}
\label{id-eq-theorem}
(UA)
For each $x,y : M$, we have $(x=_M y) ≃ \left(\Id_M␣ x␣ y \right)$.
\end{theorem}
\begin{proof}
By W-induction. Assume $a,b : U$ and $f : T a → M$ and $g  : T b → M$. Then

\begin{align*}
(\sup a␣ f)=_M (\sup b␣ g) & ≡ ∑_{α : T a ≃ T b} ∏_{x: Ta} (fx) =_M (g (α x)) \\
 \text{Induction hypothesis\ \ \ \ \ \ }&≃ ∑_{α : T a ≃ T b} ∏_{x: Ta} \Id␣ (f␣x)␣ (g (α x))\\
 \text{Definition of ${\sim}$\ \ \ \ \ }&≡ ∑_{α : T a ≃ T b}  f \sim g · α\\
 \text{Extensionality\ \ \ \ \ \ \ }& ≃ ∑_{α : T a ≃ T b}  \Id␣ f␣ (g · α)\\
 \text{Univalence \ \ \ \ \ \ \ \ }& ≃ ∑_{α : a = b} \Id ␣f ␣ (g · Tα)\\
 \text{Previous lemma\ \ \ \ \ \ }& ≃ \Id␣ (\sup a␣ f)␣ (\sup b␣g)
\end{align*}
\end{proof}

The following lemma is important with respect to constructing multisets from
logical formulas, in analogy to the comprehension axiom of set theory. We assume
that our universe has $+$, $Σ$ and $Π$-types, so if we can only prove that the
base relations ${=_M}$ and ${∈}$ also live in the universe, then we can
interpret all bounded first order formulas as families of types in $U$,
indexed by some product of $M$ with it self. Thus we have the following lemma:

\begin{lemma}
$\Id_M$ is essentially $U$-small, in the sense that for every $x,y : M$ there is
an code $ι␣ x␣ y : U$ such that $T(ι␣ x␣ y) ≃ \Id_M␣ x␣ y$.
\end{lemma}

\subsection{Extensionality}
\label{multiset-extensionality-section}

In set theory, the axiom of extensionality expresses that two sets are
considered equal if they have the same elements. More precisely, two sets $x$
and $y$ are equal if for any $z$ we have that $z ∈ x$ iff $z ∈ y$. This
formulation of extensionality fails for multisets, but a very similar
extensionality axiom may be formulated.

\textbf{The principle of extensionality for multisets:} \emph{Two multisets $x$
and $y$ are considered equal if for any $z$, the number of occurrences of $z$ in
$x$ and the number of occurrences of $z$ in $y$ are in bijective correspondence
(in our symbolism: $z ∈ x ≃ z ∈ y$).}

We will now prove a strong version of this principle for our model. The crucial
part of this is summarised in the following lemmas, concerning the fibres of
functions.

\begin{definition}
Given a function $f : A → B$  we define $\fibre{f} : B → \set$
by $\fibre{f␣b} ≔ ∑_{a:A} \Id␣(f␣a)␣b$
\end{definition}

\begin{lemma}
\label{fibre-fun}
Given function extensionality, for any $A,B,C : \set$, and functions $f : A → C$
and $g : B → C$, the following equivalence holds:

\begin{align}
\left(∑_{α : A → B} g ∘ α \sim f \right) ≃
\left(∏_{c : C} \fibre{f␣c} → \fibre{g␣c} \right)
\end{align}

\end{lemma}

\begin{proof}
 We define the maps γ and δ as follows:

\begin{align*}
&γ :\left(∑_{α : A → B} g ∘ α \sim f \right)  → \left(∏_{c : C} \fibre{f␣c} → \fibre{g␣c} \right) \\
&γ␣(α , σ)␣c␣(a , p) ≔ (σ_a , σ_a·p)
\end{align*}

\begin{align*}
&δ : \left(∏_{c : C} \fibre{f␣c} → \fibre{g␣c} \right) → \left(∑_{α : A → B} g ∘ α \sim f \right) \\
&δ␣F ≔ (λa. π₀ (F␣ (f␣ a) (a , \refl_a)) , λa.π₁ (F␣(f␣ a)␣ (a , \refl_a)))
\end{align*}

Unfolding the definitions shows that $δ␣(γ␣(α , σ)) ≡ (α , σ)$ (up to
η-reduction). \Id-induction on the fibres of $f$ shows that $γ␣(δ␣F) \sim F$. Thus,
by function extensionality, we have the desired equivalence.
\end{proof}

\begin{lemma}
\label{fibre-equiv}
Given function extensionality, for any $A,B,C : \set$, and functions $f : A → C$
and $g : B → C$, the following equivalence holds:

\begin{align}
\left(∑_{α : A ≃ B} g ∘ α \sim f \right) ≃
\left(∏_{c : C} \fibre{f␣c} ≃ \fibre{g␣c} \right)
\end{align}

\end{lemma}

\begin{proof}

The proof goes by showing that the equivalence constructed in \ref{fibre-fun}
preserves equivalences. Since being an equivalence is a $(-1)$-type, the
resulting restriction of \ref{fibre-fun} to equivalences is again an
equivalence.

Let γ and δ be as in \ref{fibre-fun}, and denote by γ' and δ' the same
construction, but with $f$ and $g$ having exchanged roles. First step is to show
that for all $α : A → B$ and $σ : g ∘ α \sim f$, if $α$ is an equivalence, then for
every $c : C$ the function $γ␣(α,σ)␣c : \fibre{f␣ c} → \fibre{g␣ c}$ is an
equivalence. Let $F_c≔γ␣(α,σ)␣c$. We construct the inverse
$F_c^{-1} ≔ γ'␣(α^{-1},σ')$,
where $σ' : f ∘ α^{-1} \sim g$ is the proof obtained by reversing
$σα^{-1} : g ∘ α ∘ α^{-1} \sim f ∘ α^{-1}$ and composing with the proof that $g ∘ α
∘ α^{-1} \sim g$. That $F_c^{-1}$ is indeed an inverse of $F_c$ can be verified by
\Id-induction on the fibres of $f$ and $g$ respectively.

We then show that for all $F$ such that $F␣c : \fibre{f␣ c} → \fibre{g␣ c}$ is
an equivalence for all $c : C$, the function $π₀ (δ␣F) : A → B$ is an
equivalence. Let $α ≔ π₀ (δ␣F)$. Its inverse is given by
$α^{-1} ≔ π₀ (δ' (F^{-1})$,
and the fact that it is an inverse of $α$ stems from the fact
that for any $a : A$ and $c : C$ and $h : \Id␣ (f␣a)␣ c$ we have that the
transport of $(a , \refl_{f a}) : \fibre{f␣ (f␣a)}$ along $h$ is $(a , h)$, and
likewise for $g$.

\end{proof}

\begin{theorem}
\label{multiset-extensionality}
(UA)
Given $x,y : M$, the following equivalence holds.

\begin{align}
 \left(x =_M y \right) ≃ ∏_{z : M} \left( z ∈ x ≃ z ∈ y \right)
\end{align}
\end{theorem}

\begin{proof}

From Theorem \ref{id-eq-theorem} we deduce that
$x ∈ (\sup ␣A␣ f) ≃ \fibre{x␣f}$.
This allows us to reformulate the above equivalence to be an
instance of Lemma \ref{fibre-equiv}.

\end{proof}

  \section{Multiset constructions}

Aczel's $V$ is a model of CZF. To mirror this we look at axioms of constructive
set theory, and attempt to find corresponding axioms for multisets. The main
observation is that definite axioms\footnote{An axiom is \emph{definite} if any set
it claims existence of is characterised uniquely.} can be systematically
changed to axioms which makes
sense for multisets, by carefully strengthening logical equivalence, $↔$, to
equivalence in type theory $≃$. A feature of this conversion is that for the
axioms below, we can retain the constructions from Aczel's $V$ when we prove that
the changed axioms hold for $M$.

The axioms we will have a look at are

\begin{itemize}
\item
  Extensionality
\item
  Restricted separation
\item
  Union Replacement
\item
  Pairing \& singletons
\item
  Infinity.
\item
  Exponentiation / Fullness
\item
  Collection.
\end{itemize}

\subsection{Restricted separation}
The axiom of restricted separation says that we can select subsets by use of
formulas, as long as they are bounded. For multisets, this corresponds to
that we may multiply the number of occurrences by the family of sets represented
by the formula, as long as the family is $U$-small.

\begin{align*}
\mathtt{(RSEP)}& & ∀ x ∃ u ∀ z \left(z ∈ u ↔ \left( z ∈ x ∧ P \right) \right)
\end{align*}
where $P$ is a restricted formula where $u$ does not occur freely in $P$.

The formulation of \texttt{RSEP} in first order logic can be translated into
type theory, given our domain $M$, replacing ↔ with ≃.

\begin{proposition}
\begin{align*}
\mathtt{(M-RSEP)}& &∏_{x : M} ∑_{u : M} ∏_{z : M} \left(z ∈ u ≃ \left( z ∈ x ∧ T␣(P␣z) \right) \right),
\end{align*}
where $P : M → U$, is a $U$-small family.
\end{proposition}

\begin{proof}
Define
\begin{align}
&\operatorname{Sep} : (M → U) → M → M \\
&\operatorname{Sep}␣ P␣ x := \sup␣ \left(∑_{i : T␣\bar x} P␣(\tilde x␣ i)\right)␣ \left(\tilde x ∘ π₀\right)
\end{align}

We must show that for all $x$ and $P$, that for every $z$ we have
$z ∈ \operatorname{Sep}␣ P␣ x ≃ z ∈ x ∧ T␣(P␣z)$.

We have:
\begin{align}
z ∈ \operatorname{Sep}␣ P␣ x
    &≡ ∑_{p : ∑_{i : T \bar x} T␣(P␣(\tilde x␣i))} \tilde x (π₀␣ p) =_M z \\
    &≃ ∑_{i : T␣\bar x} ∑_{q : T␣(P␣(\tilde x␣i))} \tilde x (π₀␣ (i,q)) =_M z \\
    &≡ ∑_{i : T␣\bar x} ∑_{q : T␣(P␣(\tilde x␣i))} \tilde x␣i =_M z \\
    &≡ ∑_{i : T␣\bar x} \left(T␣(P␣(\tilde x␣ i)) ∧ \tilde x␣i =_M z \right)\\
    &≃ ∑_{i : T␣\bar x} \left(\tilde x␣i =_M z ∧ T␣(P␣(\tilde x␣i)\right)) \\
    &≃ ∑_{i : T␣\bar x} \left(\tilde x␣i =_M z ∧ T␣(P␣z)\right) \\
    &≃ \left(∑_{i : T␣\bar x}\tilde x␣i =_M z\right) ∧ T␣(P␣z) \\
    &≡ z ∈ x  ∧ T␣(P␣z)
\end{align}

\end{proof}

\subsection{Union-replacement}

In \cite{aczelrathjen}, the authors introduce the axiom of Union-Replacement.
We use this axiom instead of separate union and replacement axioms as it seems a
more natural construction to use. For multisets it says that if we have a family
of multisets, indexed by a multiset, we can take their multiset union.

\begin{align*}
\mathtt{(UR)}& &∀ a␣ \left( ∀ x ∈ a␣ ∃ b␣ ∀ y␣ \left( y ∈ b ↔ Q(x,y) \right) → ∃ c␣ ∀ y␣ \left( y ∈ c ↔ ∃ x ∈ a␣ Q(x,y) \right) \right)
\end{align*}

Rendering this in type theory and applying the translation to multisets we get:

\begin{proposition}
\begin{align*}
\mathtt{(M-UR)}& &∏_{a : M} \left(∏_{i : T\bar a} ∑_{b:M}␣ ∏_{y:M} \left( y ∈ b ≃ Q␣(\tilde a␣i)␣y \right) \right)\\ &&  → ∑_{c : M} ∏_{y:M} \left( y ∈ c ≃ ∑_{ i : T\bar a} Q(\tilde a␣ i)␣y \right)
\end{align*}
where $Q : V → V → Set$ is any relation.
\end{proposition}

\begin{proof}

We define
\begin{align}
&\operatorname{UnionRep} : (a : M) → (\bar a → M) → M \\
&\operatorname{UnionRep}␣ a␣ f := \sup␣ \left(∑_{i : T \bar a} \overline{(f␣ i)}\right)␣ \left(λ p.\widetilde {f␣(π₀␣p)}␣ (π₁␣p)\right)
\end{align}

Let us fix $a : M$. Then, from the assumption of \\
$∏_{i : T\bar a} ∑_{b:M}␣ ∏_{y:M} \left( y ∈ b ≃ Q␣(\tilde a␣i)␣y \right)$,
we can extract $f : T\bar a → M$ such that for all $i : T\bar a$ and all $y:M$
we get $y ∈ (f␣i) ≃ Q␣(\tilde a␣i)␣y $. What we then need, is to show that for
every $y:M$ we have
$y ∈ (\operatorname{UnionRep}␣a␣f) ≃ ∑_{ i∈ T\bar a} Q(\tilde a␣ i)␣y $.

We have

\begin{align}
y ∈ (\operatorname{UnionRep}␣a␣f) 
    &≡ ∑_{p : ∑_{i : T\bar a}\overline{f␣i}} \widetilde {f␣(π₀␣p)}␣ (π₁␣p) =_M y \\
    &≃ ∑_{i : T\bar a} ∑_{j : T \overline{f␣ i}} \widetilde {(f␣i)}␣ j =_M y \\
    &≡ ∑_{i : T\bar a} y ∈ (f␣i) \\
    &≃ ∑_{i : T\bar a} Q(\tilde a␣ i)␣y
\end{align}

\end{proof}

\subsection{Singletons}

In set theory, singletons are usually constructed by pairing an element with itself.
For multisets defining singletons from pairs would not work, as the resulting
multiset would contain the element twice, not once. We therefore will prove that
our model has singletons.

If we were to have a singleton axiom in set theory, it would look like:

\begin{align*}
\mathtt{(SING)}& &∀ a␣ ∃ b␣ ∀ z␣\left( z ∈ b ↔ z = a \right)
\end{align*}
Which for our multisets becomes:

\begin{proposition}
\begin{align*}
\mathtt{(M-SING)}& &∏_{a : M}␣∑_{b:M}␣ ∏_{z:M}␣\left( z ∈ b ≃ z =_M a \right)
\end{align*}
\end{proposition}

\begin{proof}
We define
\begin{align*}
&\operatorname{Sing} : M → M \\
&\operatorname{Sing}␣ a := \sup␣ 1␣ \left(λ i.a\right)
\end{align*}

and prove that for every $a,z : M$ we have
$\left(z ∈ \operatorname{Sing}␣ a \right) ≃ \left(z =_M a\right)$.

\begin{align}
z ∈ \operatorname{Sing}␣a &≡ ∑_{i : 1} \left(λ i. a\right)␣i =_M z\\
                          &≡ ∑_{i : 1} a =_M z\\
                          &≡ 1 ∧ (a =_M z)\\
                          &≃ z =_M a
\end{align}
\end{proof}

\emph{Notation:} We will from now on use the notation $\{a\} := \operatorname{Sing}␣a$.

\subsection{Pairing}

The axiom of pairing,
\begin{align*}
\mathtt{(PAIR)}& &∀ a␣ ∀ b␣ ∃ c␣ ∀ z␣\left( z ∈ b ↔ \left(z = a ∨ z = b\right)\right),
\end{align*}
becomes

\begin{proposition}
\begin{align*}
\mathtt{(M-PAIR)}& &∏_{a : M}␣∏_{b:M}␣∑_{c:M}␣ ∏_{z:M}␣\left( z ∈ c ≃ \left( z =_M a ∨ z =_M b\right)\right)
\end{align*}
\end{proposition}

\begin{proof}
First, define

\begin{align}
&p : M → M → 2 → M\\
&p␣ a␣ b␣ (l␣*) := a \\
&p␣ a␣ b␣ (r␣*) := b
\end{align}

Then, let us define
\begin{align*}
&\operatorname{Pair} : M → M → M \\
&\operatorname{Pair}␣ a␣ b:= \sup␣ 2␣ \left(p␣a␣b\right)
\end{align*}

It remains to show that for all $a,b,z ∈ M$ we have
$z ∈ (\operatorname{Pair}␣a␣b) ≃ \left(z =_M a ∨ z =_M b \right)$.

\begin{align}
z ∈ (\operatorname{Pair}␣a␣b) &≡ ∑_{i : 2} p␣i =_M z\\
                                &≃ p␣(l␣*) =_M z ∨  p␣(r␣*) =_M z \\
                                &≡ a =_M z ∨ b =_M z \\
                                &≃ z =_M a ∨ z =_M b
\end{align}
\end{proof}

Notation: We will from now on use the notation $\{a,b\} := \operatorname{Pair}␣a␣b$.

\begin{lemma}
\label{tuple-equality}
For all $a,b,a',b' : M$,
\begin{align}
\{a\} =_M \{a'\} &≃ a =_M a' \\
\{ a,b \} =_M \{ a',b' \} &≃ \left( (a = a' ∧ b=b') ∨ (a = b' ∧ b = a'\right).
\end{align}
\end{lemma}

\paragraph{Example} \label{doubleton} \textbf{The singleton with two elements}

Observe that $(\{∅,∅\} = \{∅,∅\}) ≃ 2$. This leads to the perhaps surprising fact
that

\begin{align}
\left(\{∅,∅\} ∈ \{\{∅,∅\}\}\right) ≃ 2.
\end{align}
This might leave us feeling a bit uneasy, as this is supposed to be a singleton,
not a ``doubleton'', but we will see later (Example
\ref{singleton-example}) how to construct a multiset in which
$\{∅,∅\}$ occurs but once, and that this construction is somewhat like a
quotient of the singleton. For the time being we accept this slight anomaly as a
consequence of our rules.

\paragraph{Remark:} Using singletons, pairs and union we can construct any
finite tupling
of elements of $M$. As we see from the binary case (\ref{tuple-equality}),
the induced mapping $M^n → M$ is not an embedding.

\subsection{Ordered Pairs}

In set theory there are many equivalent ways to encode ordered pairs from
unordered pairs. The most common one, the Kuratowski encoding, defines
$\langle a,b \rangle = \{\{a\},\{a,b\}\}$. It satisfies the characteristic
property, that for all $a,b,a',b'$,

\begin{align}
\langle a,b \rangle = \langle a',b' \rangle ↔ \left(a = a' ∧ b = b'\right)
\end{align}

To understand ordered pairs of multisets, we should therefore require that for all
$a,b,a',b' : M$

\begin{align}
\langle a,b \rangle = \langle a',b' \rangle ≃ \left(a =_M a' ∧ b =_M b'\right)
\end{align}

However, this is not satisfied by mimicking the Kuratowski encoding. We can see
this by letting $a = b = a' = b = ∅$. Given this, we can calculate that
$\{\{a\},\{a,b\}\} =_M \{\{a'\},\{a',b'\}\} ≃ 2$ while
$\left(a =_M a' ∧ b =_M b'\right) ≃ 1$.

In stead of the Kuratowski encoding, we use the older definition of 
\cite{wiener1914}, $\langle a,b \rangle = \{\{\{a\},∅\},\{\{b\}\}\}$,
which harmonises with the multiset version of the characteristic property.

\begin{proof}
Observe that $\{\{a\},∅\}≠_M\{\{b\}\}$ and $\{a\}≠_M ∅$. So we get that

\begin{align}
\langle a,b \rangle =_M \langle a',b' \rangle 
   &≡ \{\{\{a\},∅\},\{\{b\}\}\} =_M \{\{\{a'\},∅\},\{\{b'\}\}\} \\
   &≃ \{\{a\},∅\} =_M \{\{a'\},∅\} ∧ \{\{b\}\} =_M \{\{b'\}\} \\
   &≃ (\{a\} =_M\{a'\} ∧ ∅ =_M ∅) ∧ (\{b\} =_M \{b'\}) \\
   &≃ a =_M a' ∧ b =_M b'
\end{align}
\end{proof}

\subsection{Cartesian products}

We obtain the cartesian product of two multisets by nesting
$\operatorname{UnionRep}$ around pairing.

\begin{definition}
Given $a,b: M$ define
\begin{align}
a × b := \operatorname{UnionRep}␣ a␣ (λ i.\operatorname{UnionRep}␣ b␣ (λ j. \langle \tilde a␣ i, \tilde b␣ j \rangle) )
\end{align}
\end{definition}
Each pairing can occur multiple times in the cartesian product. To be precise $\langle x,y \rangle ∈ \left(a× b\right) ≃ \left(x ∈ a\right) × \left(y ∈ b\right)$.

\subsection{Functions}
\label{multiset-functions}

There are several choices one could make as to what constitutes a function
between multisets. Just like in set theory, where a function between sets is it
self a set\,—\,namely a set of pairs\,—\,we should like functions between
multisets themselves to be multisets. Therefore, a given pair
$\langle x, y \rangle$ cannot occur an unbounded number of times in a function,
as we then would have problems collecting functions into exponential multisets.

The weakest notion of a function between two multisets is just a map of
occurrences. We will refer to this as a “multiset operation”. Often one consider
the stricter notion which sends equal occurrences to equal occurrences. This
notion we will denote by “multiset function”. As we will now see, we can express
both in our model.

In the model, the notion of a multiset operation between $\sup␣A␣f$ and
$\sup␣B␣g$ corresponds exactly to a map $φ : A → B$, and the stricter notion
adds the requirement that if $f␣a =_M f␣a'$ then $g␣(φ␣a) =_M g␣(φ␣a')$. The
stricter notion corresponds to functions in Aczel’s $V$, but the weaker notion is
equivalent to the stricter notion in the case of $\sup␣A␣f$ and $\sup␣B␣g$ being
sets, in the sense of $A$ and $B$ being of type level $0$ and $f$ and $g$ being
injections. Therefore, we can consider both an extension of the notion of
function to multisets.

The question we will now entertain is: How to capture these two notions in the
kind of formulas we have so far considered for other axioms? Starting with the
operations, we remind our selves that, for iterative sets, a function
$f : A → B$ is a subset of $A×B$, such that the projection down to $A$ is a
bijection. Having equivalences available in our language, we can use the
fibrewise equivalence lemma\footnote{Lemma \ref{fibre-equiv}} to describe the corresponding
situation for multisets.

\begin{definition}
We define the weak notion of a multiset function as follows.
\begin{align*}
\operatorname{Operation} : M →& M → M → Set \\
\operatorname{Operation}_{a␣b}␣f ≔ &\left(∏_z␣z ∈ f → ∑_x␣∑_y␣z=\langle x , y \rangle \right) \\ 
                                ∧ &\left(∏_x␣x ∈ a ≃ ∑_y \langle x , y \rangle ∈ f \right) \\
                                ∧ &\left(∏_y␣y ∈ b ← ∑_x \langle x , y \rangle ∈ f \right) 
\end{align*}
\end{definition}

Observe that weakening the ${≃}$ to ${↔}$ does not give the usual definition of
a function for sets, but rather that of a total binary relation. Total binary
relations form a set in classical set theory, but in CZF this is weakened to the
subset collection axiom which states that there is a set of total relations in
which every total relation has a refinement. We will later prove that the
collection of multiset operations form a multiset, and this should be seen a
form of subset collection / fullness.

\begin{definition}
\label{operation-definition}
We define the notion of a multiset function as follows.
\begin{align*}
&\operatorname{Function} : M → M → M → Set \\
&\operatorname{Function}_{a␣b}␣f ≔ \operatorname{Operation}_{a␣b}␣f ∧ ∏_{x,x',y,y'}␣&&\left(\langle x , y\rangle ∈ f ∧ \langle x' , y' \rangle ∈ f \right) \\ &&& → x = x' → y = y'
\end{align*}
\end{definition}

\subsection{Fullness, subset collection and operations}

In constructive set theory the axiom of fullness states that for each pair of
sets $a$,$b$ there is a set of total relations from $a$ to $b$ such that any
total relation has a restriction to these. The equivalent (relative to the rest
of the axioms of CZF) axiom of subset collection is a variation of fullness
which avoids the complication of using pairs to encode relations.

\begin{align*}
\mathtt{(SUB-COLL)} & &&
∀a,b∃u∀v (∀x∈a∃y∈b Q(x,y) \\
&&& → ∃z∈u (∀x∈a∃y∈z Q(x,y) ∧ ∀y∈z∃x∈a Q(x,y)))
\end{align*}

An unfortunate feature of the subset collection axiom is that it states the
existence of certain sets without defining them uniquely. Classically, the sets
of functions would satisfy the property of fullness. In fact, the requirement
that the set of functions satisfying the fullness property is equivalent to the
axiom of choice. This raises the question of what the constructive nature of this
set really is.

Some insight on the matter can be found by studying Aczel’s model of CZF in type
theory. There the underlying type of the subset collection set between
$\sup␣A␣f$ and $\sup␣B␣g$ is the function type $A → B$. In other words, the
subset collection sets are sets of operations. However, the first order language
of set theory is extensional, and thus unable to exactly pin down what an
operation is, thus the sets of which the axiom claim existence are left
indefinite by the axiom itself. In this respect, the axiom for multisets, in our
language where we have borrowed the connective ${≃}$ from type theory, stating
the existence of multisets of operations is a refinement of collection/fullness
into a definite axiom, namely \emph{exponentiation for operations}.

\subsection{Exponentiation}
\label{multiset-exp}

We define the exponential of two multisets.

\begin{definition}
Let $a , b : M$ be multisets and define

\begin{align*}
 &\operatorname{Exp}␣a␣b : M \\
 &\operatorname{Exp}␣a␣b := \sup␣ (\bar a → \bar b)␣(λf. \sup␣\bar a␣(λi. \langle \tilde a i , \tilde b (f i) \rangle))
\end{align*}
\end{definition}

Next, we formulate the exponentiation axiom for multisets and prove that there exists
a multiset in our model satisfying this axiom.

\begin{proposition}
\begin{align*}
\mathtt{(M-EXP)}& &∏_{a : M}␣∏_{b:M}␣∑_{c:M}␣ ∏_{z:M}␣\left( z ∈ c ≃ \left( \operatorname{Operation}_{a␣b} z\right)\right)
\end{align*}
\end{proposition}

\begin{proof}
Let $c$ be $\operatorname{Exp}␣a␣b$. Thus, we need to prove that for any
given $z : M$, there is an equivalence
$\operatorname{Operation}_{a␣b}␣z ≃ z ∈ \operatorname{Exp}␣a␣b$. We will
give this equivalence in two steps. (A heuristic reason for why we need to
jump through a hoop here is that
$\operatorname{Operation}_{a␣b}␣z$ has three factors while
$z ∈ \operatorname{Exp}␣a␣b$ has two (dependent ones), and we cannot
construct the equivalence factorwise.  We therefore construct a more
finely grained equivalent which maps factorwise to both.)

\textbf{Step 1.} The type $\operatorname{Operation}_{a␣b}␣z$ is equivalent
to the following data:

\begin{itemize}
\item
  $α : \bar a → ∑_{x,y:M} 〈 x , y 〉 ∈ z$
\item
  $β : ∑_{x,y:M} 〈 x , y 〉 ∈ z → \bar b$
\item
  $ε : α ∘ π₀ = \tilde a$, where
  $π₀ : ∑_{x,y:M} 〈 x , y 〉 ∈ z → M$ extracts the $x$ component.
\item
  $δ : β ∘ \tilde b = π₁$, where
  $π₁ : ∑_{x,y:M} 〈 x , y 〉 ∈ z → M$ extracts the $y$ component.
\item
  $π₀∘π₁∘π₁ : ∑_{x,y:M} 〈 x , y 〉 ∈ z → \bar z$, which extracts the
  $z$-index, is an equivalence.
\end{itemize}

This data can be succinctly expressed by the following commutative
diagram:

\begin{align}
\xymatrix{
 \bar a \ar[ddr]_{\tilde a}\ar[rr]^{α}_{≃} & & ∑_{x,y:M}〈x,y〉 ∈ z \ar[ddl]^{π₀}\ar[ddr]_{π₁}\ar[rr]^{β} & & \bar b \ar[ddl]^{\tilde b}\\
 \\
 & M & & M &
}
\end{align}

The type $\operatorname{Operation}_{a␣b}␣z$ is a product of three factors.
The first factor is equivalent to  $π₀∘π₁∘π₁ : ∑_{x,y:M} 〈 x , y 〉 ∈ z → \bar z$ being
an equivalence, since it says that all elements of $z$ are pairs. The second
factor is equivalent to the data $α$ and $ε$ above, by the fibrewise equivalence
lemma (Lemma \ref{fibre-equiv}). Similarly, the third factor is equivalent
to the data $β$ and $δ$. Thus, we conclude that $\operatorname{Operation}_{a␣b}␣z$
is indeed equivalent to the above data.

\textbf{Step 2.} The data given in step 1 is equivalent to 
$z ∈ \operatorname{Exp}␣a␣b$, which is to say that $z$ is
equal to the graph of a map $\bar a → \bar b$.

\begin{align}
\label{pairing-diagram}
\xymatrix{
  ∑_{x,y:M}〈x,y〉 ∈ z\ar[ddr]_{λ(x,y,p) → 〈x,y〉}\ar[rr]^{π₀∘π₁∘π₁} & & \bar z \ar[ddl]^{\tilde z}\\
 \\
 & M &
}
\end{align}

Given the data in step 1, define $f : \bar a → \bar b$ by $f ≔ β ∘ α$.
Since $π₀∘π₁∘π₁ : ∑_{x,y:M} 〈 x , y 〉 ∈ z → \bar z$ is an equivalence,
which makes the diagram (\ref{pairing-diagram}) commute, we derive that

\begin{align}
 z &= \sup␣\bar a␣(λi . 〈π₀␣(α␣i),π₀␣(π₁␣(α␣i))〉)\\
   &= \sup␣\bar a␣(λi . 〈\tilde a␣i,\tilde b ␣(β␣(α␣i))〉) \\
   &= \sup␣\bar a␣(λi . 〈\tilde a␣i,\tilde b ␣(f␣i)〉)
\end{align}

which is precisely that $z$ is the graph of $f$.

In the other direction, assuming that
$z = \sup␣\bar a (λi . 〈\tilde a␣i,\tilde b ␣(f␣i)〉$, we observe
that $π₀∘π₁∘π₁ : ∑_{x,y:M} 〈 x , y 〉 ∈ z → \bar z$ is in fact an equivalence,
and project the equivalence $q : \bar z ≃ \bar a$ from the assumed equality.
We then factor $f$
into $α ≔ (q∘π₀∘π₁∘π₁)^{-1}$ and $β ≔ f∘q∘π₀∘π₁∘π₁$. The fact that the rest of the
equalities of the data hold, follows from the definition of $α$ and $β$
and that diagram (\ref{pairing-diagram}) commutes.

That this construction is an equivalence is (tedious)
routine verification, from which we spare the reader.

In conclusion, combining the above two steps, we have constructed an equivalence 
$\operatorname{Operation}_{a␣b}␣z ≃ z ∈ \operatorname{Exp}␣a␣b$.

\end{proof}

\subsection{Natural numbers}
\label{natural-numbers-multiset}

The natural number axiom is straightforward to translate, and the
construction is exactly the same as in Aczel's model.

Applying the usual abbreviations –
$S␣y␣z ≡ ∀x \left(x ∈ z ↔ \left(x ∈ y ∨ x = y\right)\right)$,
which codes the relation $z$ is the successor of $y$, and
$Z␣z ≡ ∀ x\ \neg x∈z$,
coding $z$ is zero – the axiom of infinity in set theory can be expressed as:

\begin{align*}
\mathtt{(INF)}& & ∃ u ∀ z \left(z ∈ u ↔ \left(Z␣z ∨ ∃y ∈ u \ S␣y␣z \right) \right)
\end{align*}

For multisets we give similar definitions of $S$ and $Z$, in order to
define an axiom of infinitity.

\begin{align*}
S␣y␣z &≔ ∏_{x:M} \left(x ∈ z ≅ \left(x ∈ y + x = y\right)\right) \\
Z␣z &≔ ∏_{x:M}\neg x∈z
\end{align*}

\begin{proposition}
\begin{align*}
\mathtt{(M-INF)}& & ∑_{u : M} ∏_{z:M} \left(z ∈ u ≅ \left(Z␣z + ∑_{y:M} y ∈ u ∧ S␣y␣z \right) \right)
\end{align*}
\end{proposition}

\begin{proof}
The construction of a natural number object for $M$ is the same
as the construction for Aczel’s $V$. We define a sequence
of multisets $N : ℕ → M$, by

\begin{align*}
&N␣0 ≔ ∅ \\
&N␣(n+1)≔N␣n ∪ \{N␣n\}
\end{align*}

And let our natural number object be $u = \sup ␣ℕ␣N$. It
just remains to observe that $u$ satisfies the condition of $\mathtt{M-INF}$

\begin{align*}
 z ∈ u &≡ ∑_{n:ℕ} N␣n = z \\
      &≅ z = ∅ ∨ ∑_{n:ℕ} N␣(S n) = z\\
      &≅ z = ∅ ∨ ∑_{n:ℕ} (N␣n ∪ \{N␣n\}) = z\\
      &≅ Z␣z ∨ ∑_{n:ℕ} S␣(N␣n)␣z\\
      &≅ Z␣z ∨ ∑_{y:M} y∈m ∧ S␣y␣z
\end{align*}
\end{proof}

  \section{Homotopic aspects of $M$}

The previous section might seem as though not much have changed going
from the sets $V$ to the multisets $M$. In this subsection we will take a look at
what objects might be in $M$ for which, since we work with the identity type on
$M$, higher homotopies come into play. First of all, we observe that $M$ has
the same number of non-trivial levels of homotopy as $U$ has.

\subsection{Homotopy $n$-type}

Recall from the book ``Homotopy Type Theory''\footcite{hottbook} that 
types can be divided into levels, according to how many times one can
iterate the identity type on the type before it becomes trivial, in 
the sense of being contractible. A type is contractible if it has
an element, which every other element is (uniformly) equal to.
This is captured by the following definitions.

\begin{align}
\operatorname{isContractible}(X) ≔ ∑_{x:X}∏_{x':X}x'=x
\end{align}

\begin{align}
 &\mathtt{is-}(-2)\mathtt{-type}\ X = \operatorname{isContractible}(X) \\
 &\mathtt{is-}(n+1)\mathtt{-type}\ X := ∏_{x,y:X} \mathtt{is-}(n)\mathtt{-type}(\Id␣ x␣ y)
\end{align}

\begin{proposition}
$M$ has the same homotopy $n$-type as $U$.
\end{proposition}
\begin{proof}
If $M$ is homotopy $n$-type, then $U$ is also homotopy $n$-type. This follows from the fact that the following map is an embedding.

\begin{align}
&ι : U → M \\
&ι a := (\sup␣a␣(λ a. ∅))
\end{align}

On the other hand if $U$ has homotopy $n$-type, then we show by $W$-induction on 
$M$ that $M$ also has homotopy $n$-type.

Let $x = (\sup␣a␣f)$ and $y = (\sup␣b␣g)$, and consider $\Id␣ x␣ y$. 
By Lemma \ref{w-id} we know that:

\begin{align}
Id␣ x␣ y ≃ ∑_{α : \Id_U␣a␣b} \Id_{Ta → W} f␣(Bα· g)
\end{align}

From W-induction we have the induction hypothesis that the image of \(f\)
has homotopy \(n\)-type, and by Theorem 7.1.8 in the book\footcite{hottbook}.,
we know that this \(Σ\)-type also has homotopy \(n\)-type,

\end{proof}

\subsection{HITs and multisets}

If our universe has Higher Inductive Types (HITs), we can construct multisets
where the index set is a higher groupoid structure. An interesting fact is that
even if $\bar a$ is a higher groupoid, we can still have that $x ∈ a$ is
$1\text{-type}$ for all $x$ .

\begin{example}
\label{singleton-example}
In Example \ref{doubleton} we saw that the singleton
construction unexpectedly gave
singletons where the single element occurred twice, because of its non-trivial
equalities to itself in $M$. A solution to this is to take the connected
compontent of the element in $M$ as the index set of the singleton, instead of
just $1$, along with the inclusion into $M$. However, this requires the connected
component to be $U$-small, which the usual construction does not guarantee.
Adding that assumption, which we conjecture could hold in general (in homotopical
models), since $M$ is
locally $U$-small, we can construct singletons even for elements of $M$ with
non-trivial self-identities.

For any multiset $x : M$, such that there is $t : U$ which represents the
connected component of $x$, i.e. $α:T␣t ≃ ∑_{y:M}∥x=_My∥_{-1}$, we can
define a the singleton $s␣x␣t␣α := \sup␣t (π₀∘α)$. The map $π₀∘α$ is
an embedding, since $∥x=_My∥_{-1}$ is a mere proposition. It follows that
the fibres $y∈s␣x␣t␣α$ are all propositions, and in particular $x ∈ s␣x␣t␣α$
is contractible.

\end{example}

\printbibliography % print section bibliography

\end{document}